\documentclass[12pt,twoside]{amsart}
\usepackage{amssymb,amsmath,amsthm, amscd, enumerate, mathrsfs, upgreek}
\usepackage{graphicx, hhline, tikz}
\usepackage[colorlinks=true,pagebackref,hyperindex]{hyperref}
\usepackage[all]{xy}
\usepackage{textcomp}
\usepackage{color}
\usepackage[backrefs, alphabetic, initials]{amsrefs}
\usepackage{color} 
\usepackage{latexsym}
\usepackage[T1]{fontenc} 
\usepackage{fancyhdr}
\usepackage{enumerate}
\usepackage{tikz-cd}
\usepackage{array, longtable}
\usepackage{caption}
\usepackage{mathtools}
\newcolumntype{C}{>{$}c<{$}}
\newcolumntype{L}{>{$}l<{$}}

\title[Kawamata--Miyaoka type inequality]
{Kawamata--Miyaoka type inequality for varieties with nef anticanonical divisor}

\date{\today, version 0.02}

\author{Haidong Liu}
\address{School of Mathematics, Sun Yat-sen University, Guangzhou, 510275, China}
\email{liuhd35@mail.sysu.edu.cn, jiuguiaqi@gmail.com}

\DeclareMathOperator{\rank}{rank}

\DeclareMathOperator{\Pic}{Pic}

\DeclareMathOperator{\Exc}{Exc}

\DeclareMathOperator{\reg}{reg}

\DeclareMathOperator{\Cl}{Cl}

\DeclareMathOperator{\Ext}{Ext}

\DeclareMathOperator{\sm}{sm}

\newtheorem{thm}{Theorem}[section]

\newtheorem{prop}[thm]{Proposition}
\newtheorem{conj}[thm]{Conjecture}
\newtheorem{cor}[thm]{Corollary}

\theoremstyle{definition}
\newtheorem{ex}[thm]{Example}
\newtheorem{defn}[thm]{Definition}
\newtheorem{rem}[thm]{Remark}
\newtheorem*{ack}{Acknowledgments}

\newtheorem*{ai}{AI disclosure}

\makeatletter
 
 \@addtoreset{equation}{section}
\makeatother

\begin{document}

\begin{abstract}
    Let $X$ be an $\epsilon$-lc projective variety of dimension $n\geq 2$ such that $-K_X$ is nef and $0<\epsilon \leq 1$ is a real number. Then for any nef divisors $D_1,\dots,D_{n-2}$, 
    \[
         c_1(X)^{2}\cdot D_1\cdots D_{n-2}\leq \frac{2(1+\epsilon)}{\epsilon}\,\hat c_2(X)\cdot D_1\cdots D_{n-2}.
    \] 
    In particular, there exists a Kawamata--Miyaoka type inequality
    \[
        c_1(X)^n\leq \frac{2(1+\epsilon)}{\epsilon}\,\hat c_2(X)\cdot c_1(X)^{n-2}.
    \]
\end{abstract}

\maketitle 
\tableofcontents

\section{Introduction}\label{sec1}
A normal projective variety is called \emph{weak Fano} if its anticanonical divisor is nef and big, generalizing Fano varieties. The Kawamata--Miyaoka type inequality for weak Fano varieties with mild singularities was introduced and studied in a series of works \cites{ijl,liu-liu2023,liu-liu2024,jiang-liu-liu,liu2025}. So far, we have established the inequality for $\epsilon$-lc weak Fano varieties of either Picard number one or Fano index at least three:
\begin{thm}[\cite{jiang-liu-liu}*{Theorem 1.2} and \cite{liu2025}*{Remark 3.5}]\label{thm.known}
    Let $0<\epsilon \leq 1$ be a real number. Let $X$ be a $\mathbb Q$-factorial $\epsilon$-lc weak Fano variety of dimension $n\geq 2$. Suppose either $\rho(X)=1$ or $q_{\mathbb Q}(X)\geq 3$. Then we have 
    \[
         c_1(X)^{n}\leq \frac{2(1+\epsilon)}{\epsilon}\,\hat c_2(X)\cdot c_1(X)^{n-2}.
    \]
\end{thm}
These inequalities have turned out to be very effective in the classification of weak Fano varieties, see \cite{liu-liu2024}, \cite{jiang-liu-liu}*{Theorem 1.1}, \cite{jiang-liu2025}*{Theorem 1.1}, and \cite{liu2025}*{Theorem 1.3} for example. Toward the classification of weak Fano varieties of arbitrary Picard number, an unconditional Kawamata--Miyaoka type inequality would be of great significance. In this paper, we establish such an inequality in full generality.

\begin{thm}[Corollary \ref{cor.factorization}]\label{thm.KMinq}
    Let $0<\epsilon \leq 1$ be a real number. Let $X$ be an $\epsilon$-lc projective variety of dimension $n\geq 2$ such that $-K_X$ is nef. Then for any nef divisors $D_1,\dots,D_{n-2}$, we have 
    \[
         \left(\frac{2(1+\epsilon)}{\epsilon}\hat c_2(X)-c_1(X)^{2} \right)\cdot D_1\cdots D_{n-2}\geq 0.
    \]
\end{thm}

A direct consequence is the Kawamata--Miyaoka type inequality for weak Fano varieties:

\begin{cor}\label{cor.wf}
    Let $0<\epsilon \leq 1$ be a real number. Let $X$ be an $\epsilon$-lc weak Fano variety of dimension $n\geq 2$. Then for any nef divisors $D_1,\dots,D_{n-2}$, we have 
    \[
         \left(\frac{2(1+\epsilon)}{\epsilon}\hat c_2(X)-c_1(X)^{2} \right)\cdot D_1\cdots D_{n-2}\geq 0.
    \]
    In particular, when $D_1=\cdots =D_{n-2}=c_1(X)$, we have
    \[
         c_1(X)^{n} \leq \frac{2(1+\epsilon)}{\epsilon}\,\hat c_2(X)\cdot c_1(X)^{n-2}.
    \]
\end{cor}

\subsection{Basic idea}

As shown in \cites{ijl,liu-liu2023, jiang-liu-liu,liu2025}, the key step to establish the Kawamata--Miyaoka type inequality for weak Fano varieties is to study the foliation induced by the rank one maximal destabilizing subsheaf $\mathcal F$ of the tangent sheaf $\mathcal T_X$. The rank one foliation induces a birational modification $e\colon U\to X$ and a $\mathbb P^1$-fibration $p\colon U\to T$ such that $\mathcal G$ is the induced foliation by $p$ and $-K_U+\Delta=-e^*K_X$, where $\Delta$ is an $e$-exceptional $\mathbb Q$-divisor.

To prove the pseudoeffectivity of the (orbifold) second Chern class of $X$, a key step in Ou's result \cite{ou} is to show that $K_{\mathcal G}-e^*K_X$ is pseudoeffective, hence so is its pushforward $K_{\mathcal F}-K_X$. However, this positivity is not strong enough to yield the desired Kawamata--Miyaoka type inequality. Then a key observation in \cite{liu2025} is that, the assumption $\rho(X)=1$ in \cites{liu-liu2023, jiang-liu-liu} and the assumption $q_{\mathbb Q}(X)\geq 3$ in \cite{liu2025} are both to ensure that $\Exc(e)$ contains some $p$-horizontal component $\Delta'$, such that $\Delta'$ can replace part of $-e^*K_X$ on the general fibers of $p$; for instance, when $X$ has canonical singularities, we can show that $K_{\mathcal G}+a\Delta'-\frac{1}{2}e^*K_X$ is pseudoeffective for some $a>0$, and so is its pushforward $K_{\mathcal F}-\frac{1}{2}K_X$. Then this would implies the Kawamata--Miyaoka type inequality by the $\mathbb Q$-variant of Langer's inequality.

In the remaining cases, this trick does not work due to the possible absence of $p$-horizontal $e$-exceptional divisors. Our aim is therefore to find suitable substitutes for such divisors. Intrigued by the proof of the Miyaoka--Yau inequality \cite{gkp}*{Proposition 1.1}, which is the counterpart of our Kawamata--Miyaoka type inequality, we can consider the rank one maximal destabilizing subsheaf $\mathcal R$ of the canonical extension instead of the tangent sheaf. By this new approach, either we can find a $p$-horizontal $e$-exceptional divisor, or the difference between the image of $\mathcal R$ in $\mathcal T_X$ and its saturation $\mathcal L\subset \mathcal T_X$ yields a non-trivial effective Weil divisor $D$. Then such a divisor $D$ would play the same role of the $p$-horizontal $e$-exceptional divisor as above, and we are done.

\begin{rem}
    Let $X$ be a normal compact K\"ahler variety of dimension $n$ with $\epsilon$-lc singularities. If $-K_X$ is nef and big, then $X$ is already projective by Namikawa's criterion \cite{namikawa}*{Theorem 1.6}, and the Kawamata--Miyaoka type inequality is nothing but Corollary \ref{cor.wf}. So we assume that $-K_X$ is nef but not big. In this case, the same strategy together with the algebraic foliations with rationally connected leaves established in the K\"ahler setting (see \cite{ou2025}*{Theorem 1.4}, \cite{cao-paun}*{Corollary 1.5} and \cite{ijz}*{Proposition 4.3}) would provide the following Kawamata--Miyaoka type inequality
    \[
        c_1(X)^2\cdot \alpha_1\cdots\alpha_{n-2}\leq \frac{2(1+\epsilon)}{\epsilon}\,\hat c_2(X)\cdot\alpha_1\cdots\alpha_{n-2}
    \]
    for a collection of nef and big Bott--Chern classes $\alpha_i$.
\end{rem}

\subsection{Applications}

Besides those applications in \cites{liu-liu2024, jiang-liu-liu, jiang-liu2025, liu2025}, we present further applications in this paper.
The first application is a refined Kawamata--Miyaoka type inequality in the Picard number one case by using the canonical extension.  

\begin{thm}[Theorem \ref{thm.refinedKM}]
    Let $X$ be a $\mathbb Q$-factorial canonical Fano variety of dimension $n\geq 2$ and Picard number one. 
    Let $q$ be its $\mathbb Q$-Fano index. Then
    \[
    c_1(X)^n \leq 
    \begin{dcases}
    \frac{2(n+1)}{n}\, \hat{c}_2(X)\cdot c_1(X)^{n-2} & \text{if } q=1;\\
    \frac{8n}{3n-1} \,\hat{c}_2(X)\cdot c_1(X)^{n-2} & \text{if } q=2;\\
    \frac{4q^2(n-1)}{(q^2 + 2q -1)(n-1)-2} \,\hat{c}_2(X)\cdot c_1(X)^{n-2} & \text{if } q\geq 3 \text{ odd};\\
    \frac{4q^2(n-1)}{(q^2 + 2q -2)(n-1)-2} \,\hat{c}_2(X)\cdot c_1(X)^{n-2} & \text{if } q\geq 3 \text{ even}.
    \end{dcases}
    \]
\end{thm}

The following result was proved for $\mathbb Q$-Fano index $\geq 3$ in \cite{liu-liu2024}*{Theorem 1.2}, using a detailed study of the geometry of $X$ in two stubborn cases. We reprove it directly by our new Kawamata--Miyaoka type inequality.

\begin{thm}[Corollary \ref{cor.term3fold}]
    Let $X$ be a $\mathbb Q$-factorial terminal Fano $3$-fold of Picard number one. Then 
    \[
        c_1(X)^3\leq \frac{121}{41}\, c_2(X)\cdot c_1(X),
    \]
    and equality holds if and only if $X$ is isomorphic to $\mathbb P(1,2,3,5)$.
\end{thm}

The second application concerns the study of Fano varieties of coindex 4. For the smooth case, our approach is the same as \cite{floris}*{Theorem 1.2} when $n\neq 6,7$, but differs from \cite{liujie}*{Theorem 1.2} when $n=6,7$.

\begin{thm}[Theorem \ref{thm.gsn-2}]
    Let $X$ be a Gorenstein canonical Fano variety of dimension $n\geq 4$ and coindex 4 with $H$ fundamental divisor. If either $X$ is smooth or $n\leq 7$, then $h^0(X,H)\geq n-2$.
\end{thm}

Finally we use the Kawamata--Miyaoka type inequality to obtain upper bounds of the Picard numbers of Fano $3$-folds. Although the upper bound is well known for smooth Fano $3$-folds, obtaining an effective upper bound for the singular cases, even the terminal cases, remains difficult.

\begin{thm}[Corollary \ref{cor.expub}]
    Let $X$ be an $\epsilon$-lc Fano $3$-fold with Cartier index $r$. Then 
    \[
        \rho(X)\leq \left(\frac{2(1+\epsilon)}{\epsilon}36r^2(6r-1)+6r\right)c_2(X)\cdot c_1(X)-2.
    \]
\end{thm}

\begin{ack}
    The author would like to thank Jie Liu and Wenhao Ou for their invitations to CAS last year and this summer respectively, and the institute for its hospitality. During these two visits, the essential progress in \cite{liu2025} and in the present paper was made respectively. He would also like to thank Kenta Hashizume, Masataka Iwai, Chen Jiang, Jie Liu and Wenhao Ou for useful suggestions and discussions.
    
    The author is supported by the National Key Research and Development Program of China (No. 2023YFA1009801) in part and by NSFC (No. 12571048).  
\end{ack}

\begin{ai}
The author used ChatGPT-5.6 Sol for idea testing, literature search and summarization,  verification of arguments, development of proof details,  and language editing.  The paper was written by the author, who takes full responsibility for its accuracy and content.
\end{ai}

\section{Preliminaries}

Throughout this paper, we work over the complex number field $\mathbb C$. We will freely use the basic notation in \cites{kollar-mori,lazarsfeld}.

\subsection{Singularities}

Let $X$ be a normal variety such that $K_X$ is $\mathbb Q$-Cartier. Let $f\colon Y\to X$ be a resolution of singularities and write
\[
    K_Y=f^*K_X+\sum_Ea(E,X)E,
\]
where $E$ runs over all prime divisors on $Y$ and $a(E,X)\in \mathbb Q$. We say that $X$ has \emph{canonical} (resp. \emph{klt}, \emph{lc}, \emph{$\epsilon$-lc} for some fixed $0<\epsilon\leq 1$) singularities if $a(E,X)\geq 0$ (resp. $a(E,X)>-1$, $a(E,X)\geq -1$, $a(E,X)\geq \epsilon-1$) for any exceptional divisor $E$ over $X$. For simplicity, we say that $X$ is \emph{canonical}, \emph{klt}, \emph{lc}, or \emph{$\epsilon$-lc}, respectively. It is obvious that $\epsilon$-lc singularities are klt, and $1$-lc singularities are canonical.

\subsection{$\mathbb Q$-variant of Langer's inequality}\label{sub.langerineq}

Let $X$ be a klt projective variety of dimension $n\geq 2$. Then in codimension 2, $X$ has only quotient singularities and admits a $\mathbb{Q}$-variety structure (see \cite{gkpt}*{\S\,3.2 and \S\,3.6}). Let $\mathcal{E}$ be a reflexive sheaf on $X$. As in \cite{gkpt}*{\S\,3.7}, we can define the \emph{orbifold second Chern class} $\hat{c}_2(\mathcal{E})$, which is a symmetric $\mathbb{Q}$-multilinear map:
\[
 \hat{c}_2(\mathcal{E})\colon N^1(X)_{\mathbb{Q}}^{\times (n-2)} \longrightarrow \mathbb{Q},\quad (\alpha_1,\dots,\alpha_{n-2})\longmapsto \hat{c}_2(\mathcal{E})\cdot \alpha_1\cdots\alpha_{n-2}.
\]
On the other hand, as $X$ is $\mathbb{Q}$-factorial in codimension $2$, for any two $\mathbb{Q}$-divisor classes $\beta$ and $\gamma$ on $X$, the product $\beta\cdot \gamma$ is also a well-defined symmetric $\mathbb{Q}$-multilinear map:
\[
 \beta\cdot \gamma\colon N^1(X)_{\mathbb{Q}}^{\times (n-2)}\longrightarrow \mathbb{Q},\quad (\alpha_1,\dots,\alpha_{n-2})\longmapsto \beta\cdot \gamma\cdot \alpha_1\cdots\alpha_{n-2}.
\]
Then the \emph{$\mathbb{Q}$-Bogomolov discriminant} $\hat{\Delta}(\mathcal{E})$ of a  rank $r$ reflexive sheaf $\mathcal{E}$ is defined to be a symmetric $\mathbb{Q}$-multilinear map:
	\[
	\hat{\Delta}(\mathcal{E})\coloneqq 2r \hat{c}_2(\mathcal{E}) - (r-1) c_1(\mathcal{E})^2.
	\]
Let $\alpha_1,\dots, \alpha_{n-1}$ be a collection of nef classes. 
According to \cite{keel-matsuki-mckernan}*{Lemma~6.5}, for any semistable reflexive sheaf $\mathcal{E}$ with respect to $(\alpha_1,\dots, \alpha_{n-1})$ and $\alpha_1 \cdots\alpha_{n-1}\not\equiv 0$, the following \emph{$\mathbb{Q}$-Bogomolov--Gieseker inequality} holds:
\begin{equation}\label{eq.BGineq}
 \hat{\Delta}(\mathcal{E})\cdot \alpha_1\cdots \alpha_{n-2} \geq 0.
\end{equation}

For a resolution $f\colon Y\to X$ which is minimal in codimension 2, we can define the usual \emph{second Chern class} by $c_2(X)\coloneq f_*c_2(Y)$, independent of such a resolution (see \cite{cgg}*{Definition 2.3} for example). There is a gap between $c_2(X)$ and $\hat c_2(X)$, see \cite{jiang-liu-liu}*{Theorem 4.6} for example. Similarly, for a birational modification $f\colon Y\to X$, we can define the pushforward of the orbifold second Chern class $f_*\hat c_2(Y)$ by 
\[
    f_*\hat c_2(Y) \cdot \alpha_1\cdots\alpha_{n-2}\coloneq \hat c_2(Y) \cdot f^*\alpha_1\cdots f^*\alpha_{n-2},
\]
which is compatible with $\hat c_2(X)$ if $f$ is a $\mathbb Q$-factorization:

\begin{prop}\label{prop.factorization}
    Let $X$ be a klt projective variety and $f\colon Y\to X$ be a $\mathbb Q$-factorization. Then $\hat c_2(X)=f_*\hat c_2(Y)$.
\end{prop}

\begin{proof}
    By \cite{gkpt}*{Theorem 3.13.2}, we can reduce to the case where $X$ is a klt surface. Then $X$ is $\mathbb Q$-factorial, so the $\mathbb Q$-factorization $f\colon Y\to X$ is nothing but the identity.
\end{proof}

The following inequality generalizes the $\mathbb{Q}$-Bogomolov--Gieseker inequality to the non-semistable cases, and is commonly referred to as the \emph{$\mathbb Q$-variant of Langer's inequality} (see, for example, \cite{langer2004}*{Theorem 5.1} and \cite{jiang-liu-liu}*{\S~3.1} for details).

\begin{thm}\label{thm.LangerIneqII}
    Let $\mathcal{E}$ be a rank $r$ reflexive sheaf on a klt projective variety $X$ of dimension $n\geq 2$. Let $(\alpha_1,\dots,\alpha_{n-1})$ be a collection of nef classes. Then we have
	\[
	(\alpha_1\cdots \alpha_{n-2} \cdot \alpha_{n-1}^2) (\hat{\Delta}(\mathcal{E})\cdot \alpha_1\cdots \alpha_{n-2}) + r^2 (\mu_{\max} - \mu) (\mu-\mu_{\min}) \geq 0,
	\]
	where $\mu$ (resp. $\mu_{\max}$ and $\mu_{\min}$) is the slope (resp. maximal slope and minimal slope) of $\mathcal{E}$ with respect to $(\alpha_1,\dots,\alpha_{n-1})$.
\end{thm}

Let $X$ be a $\mathbb Q$-factorial klt projective variety of dimension $n\geq 2$ such that $-K_X$ is nef. Then the \emph{pseudoeffectivity} of $\hat c_2(X)$ follows directly from the generic nefness of $\mathcal T_X$ (see \cite{ou}*{Theorem 1.3}), which is given by \cite{ou}*{Corollary 1.5} when $X$ is smooth in codimension 2. We provide a detailed proof of the general case, although it should be well known to experts.

\begin{prop}[\cite{ou}*{Corollary 1.5}]\label{prop.pseffofc2}
    Let $X$ be a $\mathbb Q$-factorial klt projective variety of dimension $n\geq 2$ such that $-K_X$ is nef. Then for any nef divisors $D_1,\dots, D_{n-2}$, 
    \[
        \hat c_2(X)\cdot D_1\cdots D_{n-2}\geq 0.
    \]
\end{prop}

\begin{proof}
    By taking limits, it suffices to prove the inequality for ample divisors $D_1,\dots, D_{n-2}$. Let $H$ be an ample divisor on $X$ and $H_a\coloneq c_1(X)+aH$ for some rational number $a>0$. Consider the slope of $\mathcal T_X$ with respect to $(D_1,\dots,D_{n-2}, H_a)$. Then the generic nefness of $\mathcal T_X$ \cite{ou}*{Theorem 1.3} provides that 
    \[
        0\leq \mu_{\min}(\mathcal T_X)\leq \mu(\mathcal T_X)\leq \mu_{\max}(\mathcal T_X)\leq c_1(X)\cdot H_a \cdot D_1\cdots D_{n-2}=n\mu(\mathcal T_X),
    \]
    and hence $n^2 (\mu_{\max} - \mu) (\mu-\mu_{\min})\leq n^2\mu(\mu_{\max} - \mu)\leq (n-1)(n\mu)^2$.
    Applying Theorem \ref{thm.LangerIneqII} to $\mathcal T_X$ and $(D_1,\dots,D_{n-2}, H_a)$ yields that
    \[
        (D_1\cdots D_{n-2}\cdot H_a^2)(\hat{\Delta}(\mathcal{T}_X)\cdot D_1\cdots D_{n-2})+ (n-1)(c_1(X)\cdot H_a \cdot D_1\cdots D_{n-2})^2\geq 0,
    \]
    where $c_1(X)\cdot H_a \cdot D_1\cdots D_{n-2}=c_1(X)^2\cdot D_1\cdots D_{n-2}+ac_1(X)\cdot H\cdot D_1\cdots D_{n-2}$. As $D_1\cdots D_{n-2}\cdot H_a^2>0$, dividing by $D_1\cdots D_{n-2}\cdot H_a^2$ yields that
    \[
        2n \hat c_2(X)\cdot D_1\cdots D_{n-2}\geq (n-1)\left(c_1(X)^2\cdot D_1\cdots D_{n-2}-\frac{(c_1(X)\cdot H_a \cdot D_1\cdots D_{n-2})^2}{D_1\cdots D_{n-2}\cdot H_a^2}\right),
    \]
    and the right hand side tends to $0$ after taking limit $a\to 0$.
\end{proof}

\subsection{Basic notions of foliations}
	
    In this subsection, we gather some basic notions and facts regarding foliations on normal varieties. A more detailed explanation can be found in \cite{araujo-druel}, \cite{druel2021}*{\S\,3} and the references therein. 
    
    \begin{defn}
        A {\it foliation} on a normal variety $X$ is a non-zero coherent subsheaf $\mathcal{F}$ of the tangent sheaf $\mathcal{T}_X$ such that 
        \begin{enumerate}
    	\item $\mathcal{T}_X/\mathcal{F}$ is torsion-free, and	
    	\item $\mathcal{F}$ is closed under the Lie bracket.
        \end{enumerate} 
        The {\it canonical divisor} of a foliation $\mathcal{F}$ is any Weil divisor $K_{\mathcal{F}}$ on $X$ such that $\det(\mathcal{F})\cong \mathcal{O}_X(-K_{\mathcal{F}})$. 
     
        Let $X_{0}\subset X_{\reg}$ be the largest open subset over which $\mathcal{T}_X/\mathcal{F}$ is locally free. A \emph{leaf} of $\mathcal{F}$ is a maximal connected and immersed holomorphic submanifold $L\subset X_{0}$ such that $\mathcal{T}_L=\mathcal{F}|_L$. A leaf is called \emph{algebraic} if it is open in its Zariski closure, and a foliation $\mathcal{F}$ is said to be \emph{algebraically integrable} if its leaves are algebraic. 
    \end{defn}

    Let $\mathcal{F}$ be an algebraically integrable foliation on a normal projective variety $X$. Then there exists a commutative diagram (see \cite{druel2021}*{\S\,3.6}), which is called the \emph{family of leaves}: 

    \begin{equation}\label{eq.familyofleaves}
        \begin{tikzcd}[row sep=large, column sep=large]
    	U \arrow[r,"e"] \arrow[d,"f"]
    	& X \\
    	T
        \end{tikzcd}
    \end{equation}
    where $U$ and $T$ are normal projective varieties, $e$ is birational and $f$ is an equidimensional fibration, and the image of a general fiber $F$ of $f$ is the closure of a leaf of $\mathcal{F}$ with $\dim F=\rank \mathcal F$. If $F$ is rationally connected, then $\mathcal F$ is said to be \emph{a foliation with rationally connected general leaves}.

    The following proposition restates \cite{ou}*{Proposition 3.1}, which is a special case of \cite{druel2017}*{Proposition 4.1}. It plays an important role in \cite{liu2025}, and it will play a similar role in this paper.

    \begin{prop}\label{prop.pseff}
        Let $f\colon U\to T$ be a surjective morphism between normal projective varieties, where $U$ is $\mathbb Q$-factorial. Let $F$ be a general fiber of $f$. Let $\Delta$ be an effective $\mathbb Q$-divisor on $U$ such that 
        $(F,\Delta|_F)$ is log canonical. If $\kappa(F, K_F+\Delta|_F)\geq 0$, then $K_{\mathcal G}+\Delta$ is pseudoeffective, where $\mathcal G$ is the foliation induced by $f$. 
    \end{prop}

\subsection{The canonical extension}

We recall the canonical extension in the singular setting from \cite{gkp}*{\S 4}. Let $X$ be a normal variety and $D$ be a $\mathbb Q$-Cartier Weil divisor on $X$. Then $c_1(D)\in H^1(X,\Omega_X^1)\cong\Ext^1(\mathcal O_X, \Omega_X^1)$ corresponds to a locally split extension
\[
    0\to \Omega_X^1\to \mathcal W_D\to \mathcal O_X\to 0.
\]
By duality, we obtain a locally split extension of $\mathcal T_X$ by $\mathcal O_X$ as
\[
    0\to \mathcal O_X\to \mathcal E_D\to \mathcal T_X\to 0,
\]
whose middle term $ \mathcal E_D$ is reflexive. In the case $D=-K_X$, we call it the \emph{canonical extension} induced by $-K_X$, denoted by
\begin{equation}\label{eq.canextdef}
     0\to \mathcal O_X\to \mathcal E_X\to \mathcal T_X\to 0.
\end{equation}
By a slight abuse of language, we also call the reflexive sheaf $\mathcal E_X$ the canonical extension.

In general, the Chern character need not be additive for a short exact sequence of reflexive sheaves (cf. \cite{langer}*{\S~5.5}) unless the sequence is $\mathbb Q$-exact, i.e., exact as $\mathbb Q$-sheaves (cf. \cite{kawamata}*{\S~2} and \cite{gkpt}*{\S~3}). In particular, additivity holds for locally split exact sequences.

\begin{prop}\label{prop.c1c2}
    Let $X$ be a $\mathbb Q$-factorial klt projective variety of dimension $n\geq 2$. Then $c_1(\mathcal E_X)=c_1(X)$ and $\hat c_2(\mathcal E_X)=\hat c_2(X)$.
\end{prop}

\begin{proof}
    After a refinement of $\mathbb Q$-variety structure if necessary, we may assume that on each chart of the $\mathbb Q$-variety structure, the canonical extension is split. Then passing to the smooth covers, we obtain split exact sequences of locally free sheaves, and hence the additivity of the Chern characters. These equalities are compatible on overlaps and hence glue to the corresponding equalities of the first two Chern classes.
\end{proof}

In this paper, we only consider the well known functoriality of canonical extensions in the smooth setting, which goes back to \cite{atiyah}. For the singular setting, see \cite{gkp}*{Remark 4.3}.  

\begin{prop}[Functoriality]\label{prop.functor}
    Let $j\colon Y\hookrightarrow X_{\sm}$ be an embedding, where $Y$ is a smooth subvariety of a normal projective variety $X$.  Let $D$ be a $\mathbb Q$-Cartier Weil divisor on $X$. Then there exists a commutative diagram:
    \begin{equation}\label{eq.functor}
        \begin{tikzcd}
    	0 \arrow[r] 
    	& \mathcal O_Y \arrow[r] \arrow[d,equals]
    	& \mathcal E_{j^*D} \arrow[r] \arrow[d,"\mu"]
    	& \mathcal T_{Y} \arrow[r] \arrow[d]
    	& 0\\
    	0 \arrow[r]& \mathcal O_Y \arrow[r]& j^*\mathcal E_{D} \arrow[r]& 
        j^*\mathcal T_X\arrow[r]& 0
        \end{tikzcd}
    \end{equation}
\end{prop}

The next two propositions concern the generic nefness and generic ampleness of the canonical extension.

\begin{prop}\label{prop.gnef}
    Let $X$ be a $\mathbb Q$-factorial klt projective variety of dimension $n\geq 2$ such that $-K_X$ is nef. Then $\mathcal E_X$ is generically nef. In particular, $\mu_{\min}(\mathcal E_X)\geq 0$ with respect to any collection of nef divisors.
\end{prop}

\begin{proof}
    Consider the restriction of the canonical extension \eqref{eq.canextdef} to any Mehta--Ramanathan-general curve $C$. It follows from \cite{ou}*{Theorem 1.3} that $\mathcal T_X|_C$ is nef. Then $\mathcal E_X|_C$ is also nef as an extension of nef vector bundles is nef \cite{lazarsfeld}*{Theorem 6.2.12 (ii)}. Therefore, $\mathcal E_X$ is also generically nef.

    Let $D_1,\dots, D_{n-1}$ be a collection of nef divisors and $H$ be an ample divisor on $X$. Then for any torsion-free quotient $\mathcal Q$ of $\mathcal E_X$ and any rational number $\varepsilon>0$, the generic nefness of $\mathcal E_X$ provides that
    \[
        c_1(\mathcal Q)\cdot (D_1+\varepsilon H)\cdots (D_{n-1}+\varepsilon H)\geq 0.
    \]
    In particular, $\mu_{\min}(\mathcal E_X)\geq 0$ by taking limits.
\end{proof}

\begin{prop}\label{prop.gample}
    Let $X$ be a $\mathbb Q$-factorial klt weak Fano variety of dimension $n\geq 2$. Then $\mathcal E_X$ is generically ample.
\end{prop}

\begin{proof}
    Consider the restriction of the canonical extension \eqref{eq.canextdef} to any Mehta--Ramanathan-general curve $C$. It is an exact sequence
    \begin{equation}\label{eq.resofce}
        0\to \mathcal O_C\to \mathcal E_X|_C\to \mathcal T_X|_C\to 0,
    \end{equation}
    whose extension class in $H^1(C, \Omega_X^1|_C)$ maps to $c_1(-K_X|_C)\in H^1(C, \Omega_C^1)$ by the functoriality. The latter class is nonzero as $-K_X\cdot C>0$, and so is the extension class of \eqref{eq.resofce}. In particular, \eqref{eq.resofce} is not split. 
    
    It suffices to show that $\mu_{\min}(\mathcal E_X|_C)>0$ by the Mehta--Ramanathan theorem. Let $\mathcal Q$ be the last graded piece of the Harder--Narasimhan filtration of $\mathcal E_X|_C$, which is semistable and $\mu(\mathcal Q)\geq 0$ by Proposition \ref{prop.gnef}. If the composition $\mathcal O_C\to \mathcal E_X|_C\to \mathcal Q$ is zero, then $\mathcal Q$ is a quotient of $\mathcal T_X|_C$, hence $\mu (\mathcal Q)>0$ by the generic ampleness of $\mathcal T_X|_C$ in \cite{liu-liu2023}*{Proposition 3.6}. Let $\mathcal L\subset \mathcal Q$ be the saturation of the image of $\mathcal O_C$. Then $\mathcal O_C\hookrightarrow \mathcal L$, and hence $\deg \mathcal L\geq 0$. If $\deg \mathcal L>0$, then $\mu(\mathcal Q)\geq \deg \mathcal L>0$ as $\mathcal Q$ is semistable. So we assume that  $\deg \mathcal L=0$, i.e., $\mathcal L\simeq \mathcal O_C$. 

    If $\rank \mathcal Q=1$, then $\mathcal O_C\simeq \mathcal L\simeq \mathcal Q$. Hence $\mathcal E_X|_C\to \mathcal Q\simeq \mathcal O_C$ shows that \eqref{eq.resofce} is split, which is a contradiction. So we further assume that $\rank \mathcal Q>1$. Then there is an induced surjection $\mathcal T_X|_C\to \mathcal Q/\mathcal L$, and hence $\mu(\mathcal Q)=\frac{\rank \mathcal Q-1}{\rank \mathcal Q}\mu (\mathcal Q/\mathcal L)>0$ by \cite{liu-liu2023}*{Proposition 3.6}. 

    In any cases, we have $\mu(\mathcal Q)>0$, and the conclusion follows.
\end{proof}

\section{Kawamata--Miyaoka type inequality}\label{sec.KM}

As explained in the introduction, we consider the maximal destabilizing subsheaf of the canonical extension, rather than that of the tangent sheaf in \cites{ijl,liu-liu2023,liu-liu2024,jiang-liu-liu, liu2025}.

\begin{thm}\label{thm.main}
    Let $X$ be a $\mathbb Q$-factorial $\epsilon$-lc projective variety of dimension $n\geq 2$ such that $-K_X$ is nef and $0<\epsilon \leq 1$ is a real number. Then for any nef divisors $D_1,\dots,D_{n-2}$, we have 
    \[
         c_1(X)^{2} \cdot D_1\cdots D_{n-2}\leq \frac{2(1+\epsilon)}{\epsilon}\,\hat c_2(X)\cdot D_1\cdots D_{n-2}.
    \]
\end{thm}

\begin{proof}
    Suppose $c_1(X)^2\cdot D_1\cdots D_{n-2}=0$. Then it suffices to prove the pseudoeffectivity of $\hat c_2(X)$, which is given by Proposition \ref{prop.pseffofc2}. So we assume that $c_1(X)^2\cdot D_1\cdots D_{n-2}>0$ throughout the proof.
    Consider the canonical extension \eqref{eq.canextdef}
    \begin{equation}\label{eq.canext}
        0\to \mathcal O_X\to \mathcal E_X\xrightarrow{\iota}\mathcal T_X\to 0,
    \end{equation}
    where $\rank(\mathcal E_X)=n+1$, $c_1(\mathcal E_X)=c_1(X)$ and $\hat c_2(\mathcal E_X)=\hat c_2(X)$ by Proposition \ref{prop.c1c2}. 

    Suppose $\mathcal E_X$ is semistable with respect to $(D_1,\dots,D_{n-2}, c_1(X))$. Then the $\mathbb{Q}$-Bogomolov--Gieseker inequality \eqref{eq.BGineq} provides that 
    \[
        c_1(X)^{2}\cdot D_1\cdots D_{n-2}\leq \frac{2(n+1)}{n}\hat c_2(X)\cdot D_1\cdots D_{n-2}< \frac{2(1+\epsilon)}{\epsilon}\,\hat c_2(X)\cdot D_1\cdots D_{n-2}.
    \]
    
    So we may assume that $\mathcal E_X$ is not semistable with respect to  $(D_1,\dots,D_{n-2}, c_1(X))$. By Proposition \ref{prop.gnef}, we have $\mu_{\min}(\mathcal{E}_X)\geq 0$. As $c_1(X)^2\cdot D_1\cdots D_{n-2}>0$ and $\mu_{\max}(\mathcal{E}_X)>\mu(\mathcal{E}_X)=\frac{c_1(X)^2\cdot D_1\cdots D_{n-2}}{n+1}>0$, applying Theorem~\ref{thm.LangerIneqII} to $\mathcal E_X$ and $(D_1,\dots,D_{n-2}, c_1(X))$ yields
    \begin{align*}
        \hat{\Delta}(\mathcal{E}_X)\cdot D_1\cdots D_{n-2} + (n+1)\left(\mu_{\max}(\mathcal{E}_X)-\mu(\mathcal{E}_X)\right)
        & \geq 0,
    \end{align*}
   which implies that
    \begin{equation}\label{eq.keyineq}
        (2\hat{c}_2(X)- c_1(X)^{2})\cdot D_1\cdots D_{n-2}  + \mu_{\max}(\mathcal{E}_X) \geq 0.
    \end{equation}
    Denote by $\mathcal{F}$ the maximal destabilizing subsheaf of $\mathcal{E}_X$. Then Proposition \ref{prop.gnef} implies that 
    \begin{equation}\label{eq.genample}
        c_1(\mathcal{F})\cdot c_1(X)\cdot D_1\cdots D_{n-2}\leq c_1(X)^2\cdot D_1\cdots D_{n-2}.
    \end{equation}
    Suppose $\rank (\mathcal F)\geq 2$. Then \eqref{eq.genample} provides that
    \[
        \mu_{\max}(\mathcal{E}_X) = \mu(\mathcal{F})=\frac{c_1(\mathcal{F})\cdot c_1(X)\cdot D_1\cdots D_{n-2}}{\rank (\mathcal{F})} \leq \frac{1}{2} \, c_1(X)^2\cdot D_1\cdots D_{n-2}.
    \]
    Combining with \eqref{eq.keyineq} yields that
    \[
        c_1(X)^{2}\cdot D_1\cdots D_{n-2}\leq 4\,\hat c_2(X)\cdot D_1\cdots D_{n-2}.
    \]
    Suppose $\rank (\mathcal F)=1$. Then consider the composition $\mathcal F \hookrightarrow \mathcal E_X \to \mathcal T_X$, where $\mathcal F$ is reflexive. If $\mathcal F\to \mathcal T_X$  is zero, i.e., $\mathcal F\subset \mathcal O_X$, 
    then $c_1(\mathcal F)\cdot c_1(X)\cdot D_1\cdots D_{n-2}\leq 0$, contradicting that $\mu(\mathcal{F})=\mu_{\max}(\mathcal{E}_X)>0$.
    Hence we assume that $\mathcal F\to \mathcal T_X$ is injective.
    Let $\mathcal L\subset \mathcal T_X$ be the saturation of its image. Then $\mathcal F\hookrightarrow \mathcal L$, and hence there exists an effective Weil divisor $D$ such that $\mathcal F\simeq \mathcal L(-D)^{**}$. It follows that $c_1(\mathcal L)=c_1(\mathcal F)+D$ and $\mu(\mathcal L)\geq \mu(\mathcal F)>0$. Then by \cite{araujo-druel-2019}*{Theorem 2.19} or \cite{ou}*{Proposition 2.2} (which extend \cite{cp}*{Theorem 1.1} to singular settings), the saturated subsheaf $\mathcal L$ of $\mathcal T_X$ is a rank one algebraically integrable foliation with rationally connected general leaves, and hence induces a family of leaves as in \eqref{eq.familyofleaves}:
    \begin{equation}\label{eq.folwithrc}
    \begin{tikzcd}[row sep=large, column sep=large]
	U \arrow[r,"e"] \arrow[d,"f"]
	& X \\
	T
    \end{tikzcd}
    \end{equation}
    where the general fiber $F$ of $f$ is isomorphic to $\mathbb P^1$. Recall that $X$ has $\epsilon$-lc singularities, where $0<\epsilon\leq 1$ is a real number. Suppose that there exists an $f$-horizontal $e$-exceptional prime divisor on $U$. Then by \cite{liu2025}*{Remark 3.2 and Proposition 3.1 (3)} (where the weak Fano condition can be weaken to that $-K_X$ is nef), we have 
    \[
        (c_1(X)+(1+\epsilon)K_{\mathcal L})\cdot c_1(X)\cdot D_1\cdots D_{n-2}\geq 0.
    \]
    As $\mu_{\max}(\mathcal{E}_X)= \mu(\mathcal{F})=c_1(\mathcal{F})\cdot c_1(X)\cdot D_1\cdots D_{n-2}=(-K_{\mathcal L}-D)\cdot c_1(X)\cdot D_1\cdots D_{n-2}\leq -K_{\mathcal L}\cdot c_1(X)\cdot D_1\cdots D_{n-2}$, combining with \eqref{eq.keyineq} yields that 
    \[
        c_1(X)^{2} \cdot D_1\cdots D_{n-2}\leq \frac{2(1+\epsilon)}{\epsilon}\,\hat c_2(X)\cdot D_1\cdots D_{n-2}.
    \]
    
    So we assume that there exists no $f$-horizontal $e$-exceptional prime divisor on $U$. In this case, the rank one foliation $\mathcal L$ gives rise to an almost holomorphic map $g\colon X\dasharrow T$,  whose general fiber $C\coloneq e(F)\simeq \mathbb P^1$, and $f,g$ are isomorphic around general fibers $F$ and $C$. In particular, $C\subset X_{\sm}$ and $\mathcal L|_C=\mathcal T_C\subset \mathcal T_X|_C$, which implies that $c_1(\mathcal L)\cdot C=-K_X\cdot C=2$. We restrict \eqref{eq.canext} to $C$ and set  $\mathcal E'\coloneq (\iota|_C)^{-1}(\mathcal T_C)$. Then the functoriality of the canonical extension \eqref{eq.functor} provides that
    \begin{equation}\label{eq.exseqonC}
        0\to \mathcal O_C \to \mathcal E'=\mathcal E_{-K_X|_C}\to \mathcal T_C\to 0.
    \end{equation}
    As the extension class of \eqref{eq.exseqonC} is $-K_X|_C$ with $c_1(-K_X|_C)=-K_X\cdot C=2$, the exact sequence \eqref{eq.exseqonC} is not split. 

    We claim that $d\coloneq D\cdot C\geq 1$. As $D$ is Cartier around the general fiber $C$, we have $d\in\mathbb Z_{\geq 0}$. Suppose $d=0$. Then $\mathcal F|_C\simeq \mathcal L(-D)^{**}|_C\simeq\mathcal L|_C$. As $\mathcal F|_C\subset \mathcal E'$ by our construction, the isomorphism $\mathcal T_C\simeq \mathcal L|_C\simeq \mathcal F|_C\subset \mathcal E'$ implies that \eqref{eq.exseqonC} is split, which is a contradiction. Therefore, $d=D\cdot C\geq 1$.

    Let $D_U$ be the strict transform of the non-trivial effective Weil divisor $D$ on $U$. Then $D_U\cdot F=D\cdot C=d\geq 1$.  As in \cite{liu2025}*{Proposition 3.1}, we have
    \[
        K_U-\frac{1}{2}e^*K_X+\frac{1}{d}D_U+\Delta_1\sim_{\mathbb Q} \frac{1}{2}e^*K_X+\frac{1}{d}D_U+\Delta_2,
    \]
    where $\Delta_1$ and $\Delta_2$ are effective $f$-vertical $e$-exceptional $\mathbb Q$-divisors without common components. Then on the general fiber $F$, we have $(\frac{1}{2}e^*K_X+\frac{1}{d}D_U+\Delta_2)\cdot F=-1+1+\Delta_2\cdot F=0$. In particular, $(\frac{1}{2}e^*K_X+\frac{1}{d}D_U+\Delta_2)|_F$ is $\mathbb Q$-linearly trivial on $F\simeq \mathbb P^1$. 

    Let $H$ be an ample divisor on $U$ and $\varepsilon>0$ be a rational number. Then $-\frac{1}{2}e^*K_X+\varepsilon H$ is ample as $-e^*K_X$ is nef. Note that $D_U$ and $\Delta_1$ have no common components, as $D_U$ is a strict transform and $\Delta_1$ is $e$-exceptional. So we can choose an effective $\mathbb Q$-divisor $B$ general enough  on $U$ such that $B\sim_{\mathbb Q} -\frac{1}{2}e^*K_X+\varepsilon H$ and $(F, \frac{1}{d}D_U|_F+\Delta_1|_F+B|_F)$ is log canonical. Moreover, we have
    \[
        K_F+\frac{1}{d}D_U|_F+\Delta_1|_F+B|_F\sim_{\mathbb Q} (\frac{1}{2}e^*K_X+\frac{1}{d}D_U+\Delta_2+\varepsilon H)|_F,
    \]
    which implies that $\kappa(F, K_F+\frac{1}{d}D_U|_F+\Delta_1|_F+B|_F)\geq 0$.
    
    As in \cite{liu2025}*{Proposition 3.1}, we can further assume that $U$ is $\mathbb Q$-factorial by taking a resolution of $U$ if necessary. Then we can apply Proposition \ref{prop.pseff} to $f\colon U\to T$ and obtain that $K_{\mathcal G}+\frac{1}{d}D_U+\Delta_1+B$ is pseudoeffective, where $\mathcal G$ is the foliation induced by $f$. This implies that $K_{\mathcal G}+\frac{1}{d}D_U+\Delta_1-\frac{1}{2}e^*K_X$ is pseudoeffective by taking limit $\varepsilon\to 0$. Hence 
    \[
        K_{\mathcal L}+\frac{1}{d}D-\frac{1}{2}K_X=e_*(K_{\mathcal G}+\frac{1}{d}D_U+\Delta_1-\frac{1}{2}e^*K_X)
    \]
    is pseudoeffective. Adding the effective divisor $(1-\frac{1}{d})D$ provides that 
    \[
        -c_1(\mathcal F)-\frac{1}{2}K_X=-c_1(\mathcal L)+D-\frac{1}{2}K_X=K_{\mathcal L}+D-\frac{1}{2}K_X
    \]
    is also pseudoeffective.
    Hence $-c_1(\mathcal F)\cdot c_1(X)\cdot D_1\cdots D_{n-2}+\frac{1}{2}c_1(X)^2\cdot D_1\cdots D_{n-2}\geq 0$, and combining with \eqref{eq.keyineq} yields that
    \[
        c_1(X)^{2}\cdot D_1\cdots D_{n-2}\leq 4\,\hat c_2(X)\cdot D_1\cdots D_{n-2}. \qedhere
    \]
\end{proof}

We can drop the $\mathbb Q$-factoriality condition by pushing forward the orbifold Chern classes from a $\mathbb Q$-factorization, whose existence is given by \cite{BCHM}.

\begin{cor}\label{cor.factorization}
    Let $X$ be an $\epsilon$-lc projective variety of dimension $n\geq 2$ such that $-K_X$ is nef and $0<\epsilon \leq 1$ is a real number. Then for any nef divisors $D_1,\dots,D_{n-2}$, we have 
    \begin{equation}\label{eq.kminq}
        c_1(X)^{2} \cdot D_1\cdots D_{n-2}\leq \frac{2(1+\epsilon)}{\epsilon}\,\hat c_2(X)\cdot D_1\cdots D_{n-2}.
    \end{equation}
\end{cor}

\begin{proof}
    Let $f\colon Y\to X$ be a $\mathbb Q$-factorization. Then $Y$ is a  $\mathbb Q$-factorial $\epsilon$-lc projective variety such that $-K_Y$ is nef by \cite{kollar-mori}*{Proposition 5.20}. The conclusion follows by applying Theorem \ref{thm.main} to $Y$ with respect to $(f^*D_1,\dots, f^*D_{n-2})$ and Proposition \ref{prop.factorization}.
\end{proof}

The following example shows that the bound $\frac{2(1+\epsilon)}{\epsilon}$ is nearly optimal.

\begin{ex}
    Let $n\geq 2$ and $q\geq 3$ be two positive integers. Let $X_{n,q}\coloneq \mathbb P(1,\dots,1, nq)$, which is a $\mathbb Q$-factorial $\epsilon\coloneq \frac{1}{q}$-lc toric Fano variety of dimension $n$, as the only singularity is $\frac{1}{nq}(1,\dots,1)$. A direct calculation shows that
    \[
        c_1(X)^n=\frac{n^{n-1}(q+1)^n}{q} \quad \text{and} \quad \hat c_2(X)\cdot c_1(X)^{n-2}=(nq+\frac{n-1}{2})\frac{n^{n-2}(q+1)^{n-2}}{q}.
    \]
    It follows that
    \[
        4< b_X\coloneq \frac{c_1(X)^n}{\hat c_2(X)\cdot c_1(X)^{n-2}}=\frac{n(q+1)^2}{nq+\frac{n-1}{2}}< 2(q+1)=\frac{2(1+\epsilon)}{\epsilon}.
    \]
    Moreover, we see that $b_X\approx  \frac{1}{\epsilon}+\frac{3n+1}{2n}$ when $q\to +\infty$.
    
    In this case, we can also see from the proof of Theorem \ref{thm.main} that there must exist a rank one foliation on $X_{n,q}$ and an $f$-horizontal $e$-exceptional prime divisor on the universal family of leaves.
\end{ex}

\section{Some applications}

In this section, we present some further applications of the Kawamata--Miyaoka type inequality.

\subsection{Picard number one case revisited} 

Using the canonical extension instead, we can obtain a refinement of the Kawamata--Miyaoka type inequality in the Picard number one case, whose proof is similar to that of \cite{liu-liu2023}*{Theorem 1.2} by shifting $n$ to $n+1$. 

\begin{defn}
    Let $X$ be a canonical weak Fano variety of dimension $n$. The \emph{$\mathbb Q$-Fano index} of $X$ is defined to be
\begin{equation*}
    q=q_{\mathbb Q}(X) \coloneq \max\{m \mid -K_X \sim_{\mathbb Q}mA, \, A\in \Cl (X) \}. 
\end{equation*}
It is known that $\Cl (X)$ is a finitely generated Abelian group, so $q$ is a positive integer. 
\end{defn}

The following result is stated for canonical singularities. There should be no difficulty to generalize it to $\epsilon$-lc singularities. See also \cite{liujie}*{Theorem 1.1} for the smooth case.

\begin{thm}\label{thm.refinedKM}
    Let $X$ be a $\mathbb Q$-factorial canonical Fano variety of dimension $n\geq 2$ and Picard number one. Then
    \[
    c_1(X)^n \leq 
    \begin{dcases}
    \frac{2(n+1)}{n}\, \hat{c}_2(X)\cdot c_1(X)^{n-2} & \text{if } q=1;\\
    \frac{8n}{3n-1} \,\hat{c}_2(X)\cdot c_1(X)^{n-2} & \text{if } q=2;\\
    \frac{4q^2(n-1)}{(q^2 + 2q -1)(n-1)-2} \,\hat{c}_2(X)\cdot c_1(X)^{n-2} & \text{if } q\geq 3 \text{ odd};\\
    \frac{4q^2(n-1)}{(q^2 + 2q -2)(n-1)-2} \,\hat{c}_2(X)\cdot c_1(X)^{n-2} & \text{if } q\geq 3 \text{ even}.
    \end{dcases}
    \]
\end{thm}

\begin{proof}
    Let $A$ be the ample Weil divisor that generates the group $\Cl(X)/\sim_{\mathbb Q}$. Then we have $-K_X\sim_{\mathbb Q}qA$. Let $0=\mathcal F_0\subsetneq \mathcal F_1\subsetneq \cdots \subsetneq \mathcal F_l=\mathcal E_X$ be the Harder--Narasimhan filtration of $\mathcal E_X$ and $\mathcal G_i\coloneq (\mathcal F_i/\mathcal F_{i-1})^{**}$ be the reflexive hulls of its graded pieces. Denote by $r_i$ the rank of $\mathcal G_i$ and by $q_i$ the integer such that $c_1(\mathcal G_i)=q_iA$. By Proposition \ref{prop.gample} and by definition, we have
    \[
        \frac{q_1}{r_1}>\cdots >\frac{q_l}{r_l}>0, \quad \sum_{i=1}^lq_i=q, \quad \text{and} \quad \sum_{i=1}^lr_i=n+1.
    \]
    Applying \cite{jiang-liu-liu}*{Lemma 3.1} to $\mathcal E_X$ yields that 
    \begin{align*}
        \hat \Delta(\mathcal E_X) \cdot c_1(X)^{n-2}&\geq -\sum_{1\leq i<j\leq l}r_ir_j(\frac{q_i}{r_i}-\frac{q_j}{r_j})^2A^2\cdot c_1(X)^{n-2}\\
            &=-\sum_{1\leq i<j\leq l}(\frac{r_j}{r_i}q_i^2+\frac{r_i}{r_j}q_j^2-2q_iq_j)A^2\cdot c_1(X)^{n-2}\\
            &=(-\sum_{i=1}^l\frac{n+1-r_i}{r_i}q_i^2+\sum_{1\leq i<j\leq l}2q_iq_j)A^2\cdot c_1(X)^{n-2}\\
            &=(-(n+1)\sum_{i=1}^l\frac{q_i^2}{r_i}+q^2)A^2\cdot c_1(X)^{n-2}. 
    \end{align*}
    Dividing by $n+1$, we have
    \begin{equation}\label{eq.jll}
        2\hat c_2(X)\cdot c_1(X)^{n-2} \geq (1-\frac{1}{q^2}\sum_{i=1}^l\frac{q_i^2}{r_i})c_1(X)^n.
    \end{equation}
    If $l=1$, i.e., $\mathcal E_X$ is semistable, then the $\mathbb Q$-Bogomolov--Gieseker inequality \eqref{eq.BGineq} yields the conclusion. So we assume that $l\geq 2$. In particular, $q\geq 2$ as $q_i$ are positive integers. If $q=2$, then $\sum_{i}\frac{q_i^2}{r_i}=\frac{1}{r_1}+\frac{1}{n+1-r_1}\leq 1+\frac{1}{n}$, and the conclusion follows by \eqref{eq.jll}. Consider the case $q\geq 3$. 
    
    Suppose in addition $r_1\geq 2$. Then $\frac{1}{n-1}\leq \frac{q_l}{r_l}\leq \frac{q_i}{r_i}\leq \frac{q_1}{r_1}\leq  \frac{q-1}{2}$ for each $i$, and hence
    \[
        0\leq \sum_{i=1}^l r_i(\frac{q_i}{r_i}-\frac{1}{n-1})(\frac{q-1}{2}-\frac{q_i}{r_i})=(\frac{1}{n-1}+\frac{q-1}{2})q-\frac{(q-1)(n+1)}{2(n-1)}-\sum_{i=1}^l\frac{q_i^2}{r_i}.
    \]
    It follows that $\sum_{i=1}^l\frac{q_i^2}{r_i}\leq \frac{(q-1)^2}{2}+\frac{1}{n-1}$, and the conclusion follows by \eqref{eq.jll}. 

    Suppose $r_1=1$. Then $\mathcal F_1\hookrightarrow \mathcal T_X$, otherwise $\mathcal F_1\subset \mathcal O_X$, and hence $\mu(\mathcal F_1)\leq 0$, which is impossible. Let $\mathcal L\subset \mathcal T_X$ be the saturation of its image. Then $\mathcal F_1\hookrightarrow \mathcal L$ and there exists an effective Weil divisor $D$ such that $\mathcal F_1\simeq\mathcal L(-D)^{**}$. By \cite{liu-liu2023}*{Proposition 3.8}, we have
    \[
        \mu(\mathcal F_1)\leq \mu(\mathcal L)\leq \frac{1}{2}c_1(X)^n,
    \] 
    that is, $q_1\leq \lfloor\frac{q}{2}\rfloor$. On the other hand, if $l=2$, then $r_2=n$ and $\frac{q_2}{r_2}=\frac{q-q_1}{n}\geq \frac{q-\lfloor\frac{q}{2}\rfloor}{n}\geq \frac{2}{n}\geq \frac{1}{n-1}$; if $l\geq 3$, then $r_l\leq n-1$ and hence $\frac{q_l}{r_l}\geq \frac{1}{n-1}$.
    Then the result follows as above except for when $q$ is even and $q_1=\frac{q}{2}$. In this case, when $l=2$, we have $\sum_{i=1}^l\frac{q_i^2}{r_i}=\frac{q^2}{4}+\frac{q^2}{4n}$; when $l\geq 3$, we have $\frac{1}{n-1}\leq \frac{q_i}{r_i}\leq \frac{q}{2}-1$ for each $i\geq 2$, and hence
    \[
        0\leq \sum_{i=2}^l r_i(\frac{q_i}{r_i}-\frac{1}{n-1})(\frac{q}{2}-1-\frac{q_i}{r_i})=(\frac{1}{n-1}+\frac{q}{2}-1)\frac{q}{2}-(\frac{q}{2}-1)\frac{n}{n-1}-\sum_{i=2}^l\frac{q_i^2}{r_i}.
    \]
    It follows that $\sum_{i=1}^l\frac{q_i^2}{r_i}\leq \frac{q^2}{2}-q+1+\frac{1}{n-1}$ for $l\geq 2$, and the conclusion follows by \eqref{eq.jll}.
\end{proof}

For terminal Fano varieties of Picard number one, a direct modification of the above proof by \cite{liu-liu2023}*{(3.5)} gives that

\begin{thm}\label{thm.refinedKMinterminal}
    Let $X$ be a $\mathbb Q$-factorial terminal Fano variety of dimension $n\geq 2$ and Picard number one. Then
    \[
    c_1(X)^n \leq 
    \begin{dcases}
    \frac{2(n+1)}{n}\, \hat{c}_2(X)\cdot c_1(X)^{n-2} & \text{if } q=1;\\
    \frac{16(n-1)}{7n-9} \,\hat{c}_2(X)\cdot c_1(X)^{n-2} & \text{if } q=2;\\
    \frac{4q^2(n-1)}{(q^2 + 2q -1)(n-1)-2} \,\hat{c}_2(X)\cdot c_1(X)^{n-2} & \text{if } q\geq 3 .\\
    \end{dcases}
    \]
    Moreover, if $q=2$ and $n\leq 3$, then $\mathcal E_X$ is semistable.
\end{thm}

Immediately, Theorem \ref{thm.refinedKMinterminal} together with \cite{liu-liu2024}*{Lemma 2.2} provides an unconditional generalization of \cite{liu-liu2024}*{Theorem 1.2} without using rank two foliation and the Sarkisov link.

\begin{cor}\label{cor.term3fold}
    Let $X$ be a $\mathbb Q$-factorial terminal Fano $3$-fold of Picard number one. Then 
    \[
        c_1(X)^3\leq \frac{121}{41}\, c_2(X)\cdot c_1(X),
    \]
    and equality holds if and only if $X$ is isomorphic to $\mathbb P(1,2,3,5)$.
\end{cor}

\subsection{Fano varieties of coindex 4}

The application of the Kawamata--Miyaoka type inequality to the study of singular Fano varieties of coindex 4 has been considered by Jie Liu in the preparation of \cite{liu-liu2023} with the author, while at that moment we can only treat the Picard number one case.  In this subsection, we study the arbitrary Picard number case.

Let $X$ be a Gorenstein canonical Fano variety of dimension $n$. The \emph{index} of $X$ is defined to be
\[
    i_X\coloneq \max\{m\in\mathbb Z\mid -K_X\sim mH, \text{ where $H$ is an ample Cartier divisor}\}.
\]
The \emph{coindex} of $X$ is $n+1-i_X$. An ample Cartier divisor $H$ on $X$ with $-K_X\sim i_XH$ is called the \emph{fundamental divisor} of $X$. It is well known that $1\leq i_X\leq n+1$, and the classifications of cases of coindex $\leq 3$ have been largely settled down by the work of Kobayashi--Ochiai, Fujita, Mukai and others. Now we turn to the case of coindex $4$, equivalently, $i_X=n-3$.

\begin{thm}\label{thm.gsn-2}
    Let $X$ be a Gorenstein canonical Fano variety of dimension $n\geq 4$ and coindex 4 with $H$ fundamental divisor. If $n\leq 7$, then $h^0(X,H)\geq n-2$. 
\end{thm}

\begin{proof}
    By the Kawamata--Viehweg vanishing and \cite{floris}*{Lemma 3.1}, we have 
    \[
        h^0(X,H)=\chi(X,H)=\frac{H^n}{24}(-n^2+7n-8)+\frac{c_2(X)\cdot H^{n-2}}{12}+n-3,
    \]
    where $c_2(X)=f_*c_2(Y)$ for a crepant terminalization $f\colon Y\to X$. Then by \cite{jiang-liu-liu}*{Theorem 4.6} after taking intersections of hypersurfaces and by the Kawamata--Miyaoka type inequality \eqref{eq.kminq}, we have
    \[
        c_2(X)\cdot H^{n-2}\geq \hat c_2(X)\cdot H^{n-2}\geq \frac{1}{4}(n-3)^2H^n.
    \]
    It follows that 
    \begin{equation}
        h^0(X,H)\geq \frac{(n-1)(7-n)}{48}H^n+n-3\geq \frac{(n-1)(7-n)}{48}+n-3.
    \end{equation}
    So when $n\leq 6$, we have $ h^0(X,H)\geq n-2$. When $n=7$, we have $ h^0(X,H)\geq 4$. Applying \cite{floris}*{Theorem 1.1} to a general element $Y\in |H|$ provides that $Y$ is a Gorenstein canonical Fano variety of dimension 6 with $-K_Y\sim (-K_X-H)|_Y=3H|_Y$. Then the above result shows that $h^0(Y, H|_Y)\geq 4$. Therefore, the induced long exact sequence of
    \[
        0\to \mathcal O_X\to \mathcal O_X(H)\to \mathcal O_Y(H)\to 0
    \]
    implies that $h^0(X, H)=h^0(Y, H|_Y)+1\geq 5$.
\end{proof}

Combining with \cite{park}*{Theorem 1.1 (1)}, we immediately obtain some general elephants in dimension at most $7$:

\begin{cor}
    Let $X$ be a Gorenstein canonical Fano variety of dimension $4\le n\leq  7$ and coindex 4 with $H$ fundamental divisor.  Then a general member of $|mH|$ has only canonical singularities for every integer $m\geq 1$.
\end{cor}

\begin{rem}
    When $X$ is smooth, $h^0(X,H)\geq n-2$ has already been obtained by Floris \cite{floris}*{Theorem 1.2} for $n\neq 6, 7$ and by Jie Liu \cite{liujie}*{Theorem 1.2} for $n=6,7$  using classification results. Here  Theorem \ref{thm.gsn-2} provides an alternative proof for $n=6,7$ without using classifications.
\end{rem}
    
\subsection{Picard numbers of Fano $3$-folds}
Besides its applications to upper bounds of degrees and Fano indices, the Kawamata--Miyaoka type inequality can also be applied to find some effective upper bound for the Picard numbers of Fano $3$-folds.

\begin{thm}\label{thm.f(m)}
    Let $X$ be an $\epsilon$-lc Fano $3$-fold. Let $m$ be a positive integer such that $-mK_X$ is Cartier and $|-mK_X|$ is basepoint-free. Then 
    \[
        \rho(X)\leq f(m)=\left(m+\frac{2(1+\epsilon)}{\epsilon}m^2(m-1)\right)c_2(X)\cdot c_1(X)-2.
    \]
\end{thm}
\begin{proof}
    Let $S$ be a general member of $|-mK_X|$, which is a klt surface by Bertini's theorem. 
    Consider the exact sequence
    \[
        0\to \mathcal O_X(-S)\to \mathcal O_X\to \mathcal O_S\to 0.
    \]
    As $q(X)\coloneq h^1(X,\mathcal O_X)=0$ and $h^2(X,\mathcal O_X(-S))=0$ by the Kawamata--Viehweg vanishing, we have $q(S)\coloneq h^1(S,\mathcal O_S)=0$ from the induced long exact sequence. Applying the Grothendieck--Lefschetz theorem \cite{rs}*{Theorem 1} to $|-mK_X|$ on $X$ yields that $\Pic(X)\hookrightarrow\Cl(X)\hookrightarrow \Cl(S)$. It follows that  
    \begin{equation}\label{eq.rs}
        \Pic(X)\hookrightarrow\Pic(S)\hookrightarrow\Cl(S)
    \end{equation}
    and hence $\rho(X)\leq \rho(S)$ as $q(X)=q(S)=0$.
    
    Let $f\colon Y\to X$ be a resolution which is minimal in codimension 2 and $f|_T\colon T\to S$ is a resolution of S, where $T=f^*S=f^{-1}S$. As $S$ has rational singularities, we have $q(T)=q(S)=0$. It follows that
    \begin{equation}\label{eq.upofrhoS}
        \rho(S)\leq \rho(T)\leq h^{1,1}(T)=c_2(T)-2-2p_g(T)\leq c_2(T)-2.
    \end{equation}
   
    Moreover, from the exact sequence
    \[
        0\to \mathcal T_T\to \mathcal T_Y|_T\to \mathcal O_T(T)\to 0
    \]
    and the projection formula, we have 
    \[
        mc_2(X)\cdot c_1(X)=c_2(X)\cdot S=c_2(Y)\cdot T=c_2(\mathcal T_Y|_T)=c_2(T)+c_1(T)\cdot T|_T.
    \]
    It follows that 
    \begin{equation}\label{eq.calofc2T}
        c_2(T)=mc_2(X)\cdot c_1(X)+(K_Y+T)\cdot T^2=mc_2(X)\cdot c_1(X)+m^2(m-1)c_1(X)^3.
    \end{equation}
    Combining \eqref{eq.rs}-\eqref{eq.calofc2T} together, we have 
    \[
        \rho(X)\leq \rho(S)\leq c_2(T)-2=mc_2(X)\cdot c_1(X)+m^2(m-1)c_1(X)^3-2.
    \]
    Then the conclusion follows from $\hat c_2(X)\cdot c_1(X)\leq c_2(X)\cdot c_1(X)$ as in \cite{jiang-liu-liu}*{Theorem 4.6} and the Kawamata--Miyaoka type inequality \eqref{eq.kminq}.
\end{proof}

\begin{rem}
    We can improve each inequality in \eqref{eq.upofrhoS} better for explicit cases, in particular, for the cases with canonical singularities. We can also use $\hat c_2(X)\cdot c_1(X)$ to bound $\rho(X)$ rather than $c_2(X)\cdot c_1(X)$ to obtain sharper upper bounds. 
\end{rem}

In the rest of this subsection, we state more on the threshold of the basepoint-freeness of $-K_X$. An effective basepoint-freeness result of Angehrn--Siu type is known to hold up to quasi-log canonical singularities (see \cites{angehrn-siu, kollar, fujino2010,liu2019} and the references therein). In this paper, we apply Fujino's log canonical version \cite{fujino2010}. 


\begin{prop}\label{prop.bpf}
    Let $X$ be an lc Fano $3$-fold and $r$ be the Cartier index of $X$ such that $-rK_X$ is Cartier. Then $-6rK_X$ is basepoint-free.
\end{prop}

\begin{proof}
    
    Let $M\coloneq -6rK_X$ and $N\coloneq M-K_X=(6+\frac{1}{r})(-rK_X)$, which are both ample. Let $c(1)=c(2)=c(3)=6$. Then
    \[
        N^{\dim Z}\cdot Z\geq (6+\frac{1}{r})^{\dim Z}(-rK_X|_Z)^{\dim Z}\geq (6+\frac{1}{r})^{\dim Z}>6^{\dim Z}=c(\dim Z)^{\dim Z}
    \]
    for any subvariety $Z\subseteq X$, and 
    \[
       \sum^{3}_{k=1}\frac{k}{c(k)}=\frac{1}{6}+\frac{2}{6}+\frac{3}{6}=1.
    \]
    It follows from \cite{fujino2010}*{Theorem 1.1} that $M$ is basepoint-free. 
\end{proof}

\begin{rem}
    If $X$ is Gorenstein and canonical, then $-2K_X$ is basepoint-free by \cite{liujie}*{Theorem 1.6} and it is sharp. If $X$ is $n$-dimensional, then $-\frac{n(n+1)}{2}rK_X$ is basepoint-free as in the proof of Proposition \ref{prop.bpf}.
\end{rem}

As a direct consequence of Theorem \ref{thm.f(m)} and Proposition \ref{prop.bpf}, we obtain an upper bound for $\epsilon$-lc Fano $3$-folds. Note that for canonical Fano $3$-folds, we have $r\leq 840$ by \cite{chen-jiang}*{Proposition 2.4} and $rc_2(X)\cdot c_1(X)\leq 2489$ by a computer search.

\begin{cor}\label{cor.expub}
    Let $X$ be an $\epsilon$-lc Fano $3$-fold with Cartier index $r$. Then 
    \[
        \rho(X)\leq \left(\frac{2(1+\epsilon)}{\epsilon}36r^2(6r-1)+6r\right)c_2(X)\cdot c_1(X)-2.
    \]
    Moreover, if $X$ has only canonical singularities, then
    \[
        \rho(X)\leq 2489\left(144r(6r-1)+6\right)-2<10^{12}.
    \]
    
\end{cor}

\end{document}